\documentclass[12pt,twoside]{amsart}
\usepackage[margin=3cm]{geometry}
\usepackage[colorlinks=false]{hyperref}
\usepackage[english]{babel}
\usepackage{graphicx,titling}
\usepackage{float}
\usepackage{amsmath,amsfonts,amssymb,amsthm}
\usepackage{lipsum}
\usepackage[T1]{fontenc}
\usepackage{fourier}
\usepackage{color}
\usepackage[latin1]{inputenc}
\usepackage{esint}
\usepackage{caption}
\usepackage{ccicons}

\makeatletter
\def\blfootnote{\xdef\@thefnmark{}\@footnotetext}
\makeatother

\newcommand\ccnote{
    \blfootnote{\copyright\,\, Bo'az Klartag}
    \blfootnote{\ccLogo\, \ccAttribution\,\, Licensed under a \href{https://creativecommons.org/licenses/by/4.0/}{Creative Commons Attribution License (CC-BY)}.}
}

\usepackage[export]{adjustbox}
\numberwithin{equation}{section}
\usepackage{setspace}

\renewcommand{\leq}{\leqslant}

\renewcommand{\geq}{\geqslant}
\renewcommand{\mathbb}{\varmathbb}
\usepackage{fancyhdr}
\newtheorem{theorem}{Theorem}[section]
\newtheorem{lemma}[theorem]{Lemma}
\newtheorem{corollary}[theorem]{Corollary}
\newtheorem{proposition}[theorem]{Proposition}
\newtheorem{remark}[theorem]{Remark}
\usepackage{latexsym, amsmath, amsfonts, amsthm, amssymb, physics}
\usepackage{mathrsfs, url}

\def \vs {\smallskip}

\def \id {{\rm Id}}

\def \RR {\mathbb R}

\def \EE {\mathbb E}

\def \eps {\varepsilon}
\def \vphi {\varphi}

\def \cF {\mathcal F}

\address{
	Bo'az Klartag, Department of Mathematics,
	Weizmann Institute of Science,
	Rehovot, Israel}
\email{boaz.klartag@weizmann.ac.il}

\begin{document}

\thispagestyle{empty}

\begin{minipage}{0.28\textwidth}
\begin{figure}[H]
\includegraphics[width=2.5cm,height=2.5cm,left]{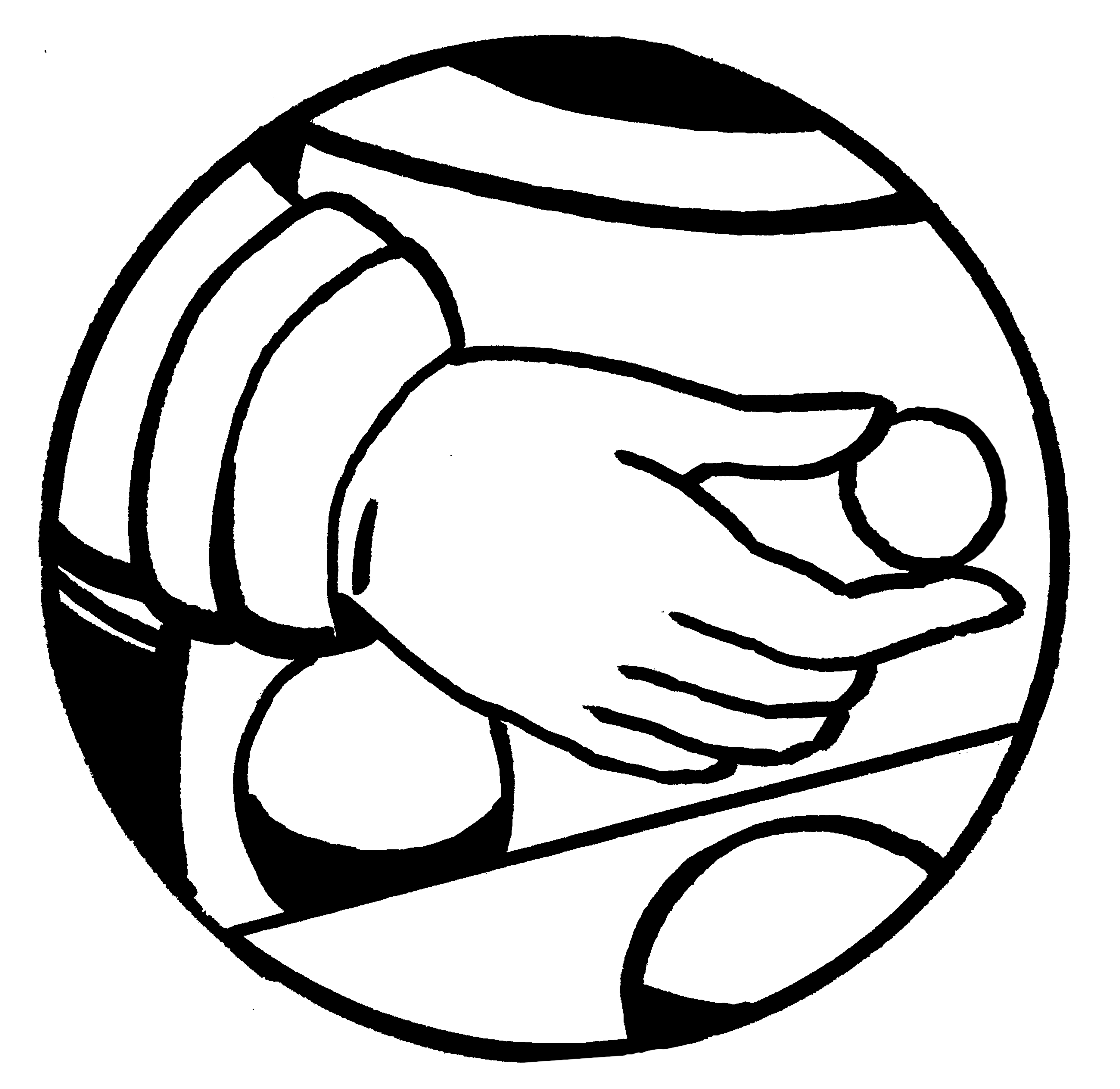}
\end{figure}
\end{minipage}
\begin{minipage}{0.7\textwidth}
\begin{flushright}
Ars Inveniendi Analytica (2023), Paper No. 4, 17 pp.
\\
DOI 10.15781/jsjy-0b06
\\
ISSN: 2769-8505
\end{flushright}
\end{minipage}

\ccnote

\vspace{1cm}


\begin{center}
\begin{huge}
\textit{Logarithmic bounds for isoperimetry and slices of convex sets}


\end{huge}
\end{center}

\vspace{1cm}


\begin{center}
{\large{\bf{Bo'az Klartag}}} \\
\vskip0.15cm
\footnotesize{Weizmann Institute of Science}
\end{center}

\vspace{1cm}


\begin{center}
\noindent \em{Communicated by Emanuel Milman}
\end{center}
\vspace{1cm}


\noindent \textbf{Abstract.} \textit{We prove that
	the Bourgain slicing conjecture and the Kannan-Lov\'asz-Simonovits (KLS) isoperimetric conjecture
	in $\RR^n$ hold true up to a factor of $\sqrt{\log n}$. A new ingredient used in the proof
	is an improved log-concave Lichnerowicz inequality.}
\vskip0.3cm

\noindent \textbf{Keywords.} Isoperimetric inequalities, hyperplane sections, high-dimensional convex sets.
\vspace{0.5cm}


\section{Introduction}

Bourgain's slicing problem asks whether any convex body $K \subseteq \RR^n$ of volume
one admits a hyperplane $H \subseteq \RR^n$ such that
$$ Vol_{n-1}(K \cap H) > c, $$
for a universal constant $c > 0$.
Here $Vol_{n-1}$ stands for $(n-1)$-dimensional volume and a hyperplane is a one-codimensional affine subspace.
We refer the reader to Klartag and Milman \cite{KM} and references therein for background on this problem, its equivalent formulations
and its applications. For $n \geq 2$ define
$$ \frac{1}{L_n} = \inf_{K \subseteq \RR^n} \sup_{H \subseteq \RR^n} Vol_{n-1}(K \cap H), $$
where the infimum runs over all convex bodies $K \subseteq \RR^n$ of volume one, and the supremum  runs over all hyperplanes $H \subseteq \RR^n$. Thus, a convex body of volume one in $\RR^n$ has a hyperplane section whose $(n-1)$-dimensional volume is at least $1 / L_n$,
and Bourgain's slicing problem asks whether $L_n < C$ for a universal constant $C > 0$.

\vs
For decades the best known bounds have been $L_n \leq C n^{1/4} \log n$ proven in Bourgain \cite{B1, B2} and $L_n \leq C n^{1/4}$ proven in \cite{K_quarter}.
Two years ago, a breakthrough by Chen \cite{chen}  led to the bound $L_n \leq \exp( C \sqrt{ \log n \cdot \log \log n} )$
which was subsequently improved to $L_n \leq C \log^4 n$ by Klartag and Lehec \cite{KL}, to $L_n \leq C \log^{2.223...} n$
by Jambulapati, Lee and Vempala \cite{JLV} and to $L_n \leq C \log^{2.082...}$ by Lehec (personal communication).
In this paper we obtain a further improvement:

\begin{theorem} For $n \geq 2$, $$ L_n \leq C \sqrt{\log n}, $$
	where $C > 0$ is a universal constant.
	\label{thm1}
\end{theorem}

A probability density $\rho: \RR^n \rightarrow [0, \infty)$ is {\it log-concave} if its support $\{ x \in
\RR^n \, ; \, \rho(x) > 0 \}$ is a convex set and $\log \rho$ is concave in the support of $\rho$.
A random vector $X$ in $\RR^n$ is log-concave if it is supported in an affine
subspace and has a log-concave density in this subspace. The uniform distribution on a convex body is log-concave,
as well as the Gaussian distributions. A random vector $X = (X_1,\ldots,X_n) \in \RR^n$ is {\it isotropic}
if it has finite second moments and for $i,j=1,\ldots,n$,
$$ \EE X_i = 0 \qquad \text{and} \qquad \EE X_i X_j = \delta_{ij} $$
where $\delta_{ij}$ is Kronecker's delta.
When we say that a Borel probability measure $\mu$ on $\RR^n$ is log-concave or is isotropic, we mean that the random vector $X$ with law $\mu$ has the corresponding property.
When $\mu$ is isotropic and has a log-concave density $\rho$, for any hyperplane $H \subseteq \RR^n$ passing through the origin,
\begin{equation} \frac{1}{\sqrt{12}}  \leq \int_{H} \rho \leq \frac{1}{\sqrt{2}}.
\label{eq_1708} \end{equation}
This is proven in Hensley \cite{hensley} and Fradelizi \cite{fradelizi}, using the Pr\'ekopa-Leindler inequality and one-dimensional analysis.
Moreover, if $H^+ \subseteq \RR^n$ is a half-space whose boundary is $H$ then Gr\"unbaum's theorem \cite{G}
states that
\begin{equation}
\frac{1}{e} \leq \mu(H^+) \leq 1 - \frac{1}{e}. \label{eq_1707}
\end{equation}
Thus the boundary measure of any such half-space has the same order of magnitude
as the measure of the half-space itself. In fact, it suffices to assume that $\mu$ is centered for (\ref{eq_1707}) to hold true, with no need for isotropicity.

\vs The {\it Poincar\'e constant} of a random vector $X$ in $\RR^n$, denoted by $C_P(X)$, is the infimum over all $C \geq 0$ such that for any locally-Lipschitz function $f: \RR^n \rightarrow \RR$
satisfying $\EE |\nabla f(X)|^2 < \infty$,
\begin{equation} Var( f(X) ) \leq C \cdot \EE |\nabla f(X)|^2. \label{eq_2154} \end{equation}
Given a probability measure $\mu$ in $\RR^n$ we write $C_P(\mu) = C_P(X)$ where $X$ is the random vector with law $\mu$.
When $\mu$ has a log-concave density $\rho$, its {\it isoperimetric constant}
is
$$ \frac{1}{\psi_{\mu}} = \inf_{A \subseteq \RR^n} \left\{ \frac{\int_{\partial A} \rho}{\min \{ \mu(A), 1 - \mu(A) \}} \right\} $$
where the infimum runs over all open sets $A \subseteq \RR^n$ with smooth boundary
satisfying  $0 < \mu(A) < 1$. This isoperimetric constant is finite and positive, as proven in Bobkov \cite{bobkov}.
The inequalities of Cheeger \cite{cheeger} and Buser and Ledoux \cite{buser, ledoux} imply that
the Poincar\'e constant and the isoperimetric constant are closely related. Namely,
\begin{equation}
\frac{1}{4} \leq \frac{\psi_{\mu}^2}{C_P(\mu)} \leq \pi.
\label{eq_1103}
\end{equation}
The  value of the numerical constant on the right-hand side of (\ref{eq_1103}) was found by De Ponti and Mondino \cite{DM} using the technique from Ledoux \cite{ledoux}.
We define
\begin{equation} \psi_n = \sup_{\mu} \psi_{\mu} \label{eq_1104} \end{equation}
where the supremum runs over all isotropic, log-concave probability measures in $\RR^n$.
The Kannan-Lov\'asz-Simonovits (KLS) conjecture \cite{KLS} suggests that $\psi_n < C$ for a universal constant $C > 0$. If correct
then it would imply that the most efficient way to partition a convex body into two pieces of equal mass
so as to minimize their interface is a hyperplane bisection, up to a universal constant.

\vs The fact that the KLS conjecture implies Bourgain's slicing conjecture
was  announced by Ball in 2003; the proof, using the heat equation and yielding
the estimate $L_n \leq \exp(C \psi_n^2)$,
was published in the work of  Ball and Nguyen \cite{BN}.
In the meantime, Eldan and Klartag \cite{EK} used the logarithmic Laplace transform and proved the  bound
\begin{equation}
L_n \leq C \psi_n,
\label{eq_1132}
\end{equation}
where $C > 0$ is a universal constant.
Substantial  progress towards the KLS conjecture
started roughly a decade ago, with Eldan's stochastic localization method \cite{Eldan1}.
By utilizing this method, Lee and Vempala \cite{LV} proved the bound $\psi_n \leq C n^{1/4}$, that was improved
to $\psi_n \leq \exp( C \sqrt{ \log n \cdot \log \log n} )$ by Chen \cite{chen},
to $\psi_n \leq C \log^5 n$ by Klartag and Lehec \cite{KL}, to $\psi_n \leq C \log^{3.223...} n$
by Jambulapati, Lee and Vempala \cite{JLV} and to $\psi_n \leq C \log^{3.082...}$ by Lehec (personal communication).
Theorem \ref{thm1} follows  from (\ref{eq_1132}) and the following:

\begin{theorem} For $n \geq 2$, $$ \psi_n \leq C \sqrt{\log n}, $$
	where $C > 0$ is a universal constant.
	\label{thm2}
\end{theorem}

Already in the  case of unconditional convex bodies,
Theorem \ref{thm2} improves upon
the previously known bound. That bound was obtained in \cite{K_uncond}.

\vs Our  new ingredient is an improved log-concave Lichnerowicz inequality.
We say that a probability measure $\mu$ in $\RR^n$ is {\it $t$-uniformly log-concave}, for $t \geq 0$, if
the probability measure $\nu$ whose density with respect to $\mu$ is proportional to $e^{t |x|^2/2}$ exists and is log-concave.
The {\it log-concave Lichnerowicz inequality} states that for any $t > 0$ and a $t$-uniformly log-concave probability measure $\mu$
in $\RR^n$,
\begin{equation}   C_P(\mu) \leq 1/t. \label{eq_154_} \end{equation}
The name of this  inequality stems from an analogy to investigations by Lichnerowicz  in Riemannian geometry, see \cite{BGL, ledoux1}.
Inequality (\ref{eq_154_}) also follows  from the Brascamp-Lieb inequality \cite{BL}.
Write $Cov(\mu) = (Cov_{ij}(\mu))_{i,j=1,\ldots,n} \in \RR^{n \times n}$ for the covariance matrix of the log-concave probability measure $\mu$, defined via
$$ Cov_{ij}(\mu) = \int_{\RR^n} x_i x_j d \mu(x) - \int_{\RR^n} x_i d \mu(x) \int_{\RR^n} x_j d \mu(x). $$
The covariance matrix is well-defined, since a log-concave probability density decays exponentially at infinity (e.g. \cite[Lemma 2.2.1]{BGPP}).
The covariance matrix is a symmetric, positive semi-definite matrix, and its operator norm is denoted by $\| Cov(\mu) \|_{op}$.

\begin{theorem}[``improved log-concave Lichnerowicz inequality''] Let $t > 0$ and let $\mu$ be a $t$-uniformly log-concave probability measure in $\RR^n$.
	Then,
	\begin{equation}
	\label{eq_1518}
	C_P(\mu) \leq \sqrt{ \frac{\| Cov(\mu) \|_{op}}{t} } \leq \frac{1}{t}.
	\end{equation}
	\label{prop_1944}
\end{theorem}

The bound in (\ref{eq_1518}) is the geometric average of the Lichnerowicz bound (\ref{eq_154_}) with the KLS conjectural bound
$C_P(\mu) \lesssim \| Cov(\mu) \|_{op}$, where $A \lesssim B$ is a shorthand for $A \leq C \cdot B$ where $C > 0$ is a universal constant. Theorem \ref{prop_1944} is proven in Section \ref{sec2} by using
the Bochner formula in order to analyze the first non-trivial eigenfunction of the Laplacian associated with $\mu$. The proof of Theorem \ref{thm2} is given
in Section \ref{sec3}, and it combines Theorem  \ref{prop_1944} with some
properties of Eldan's stochastic localization.

\vs
Our usage of stochastic localization is relatively ``soft'' compared to recent works in this field.
Stochastic processes are crucial  in
this proof only as a tool for bounding  the covariance process up until a certain time determined by $\psi_n$, see the bound (\ref{eq_1732}) below.
It could be interesting to try and replace the role of stochastic processes in this argument with differentiations along the heat flow, integrations by parts, and
$3$-tensor analysis.

\vs We write $\nabla^2 u$ for the Hessian matrix
of the function $u:\RR^n \rightarrow \RR$.
The  scalar product of $x,y \in \RR^n$ is $x \cdot y = \langle x, y \rangle = \sum_i x_i y_i$.
The Euclidean norm of $x \in \RR^n$ is $|x| = \sqrt{x \cdot x}$, and $S^{n-1} = \{ x \in \RR^n \, ; \, |x| = 1 \}$
is the unit sphere centered at the origin. A smooth function is $C^{\infty}$-smooth,
and $\log$ stands for the natural logarithm. We write $C, c, \tilde{C}, \tilde{c}, \bar{C}$ etc. to denote various positive universal constants whose value may change from one line to the next.

\vs {\it Acknowledgements.} I am grateful to Richard Gardner, Joseph Lehec, Emanuel Milman, Sasha Sodin and Ramon van Handel
for their comments on an earlier version of this paper and for pointing me to references of which I was previously unaware. Supported by a grant from the Israel Science Foundation (ISF).

\section{Eigenfunctions with a preferred direction}
\label{sec2}

Let $\mu$ be a probability measure in $\RR^n$ with a log-concave probability density $\rho$.
We say that $\mu$ is {\it regular} if its density $\rho$ is smooth and positive in $\RR^n$
and the following two requirements hold:
\begin{enumerate}
	\item[(i)] There exists $\eps > 0$ such that for all $x \in \RR^n$, denoting $\psi = -\log \rho$,
	\begin{equation}
	\eps \cdot \id \leq \nabla^2 \psi(x) \leq \frac{1}{\eps} \cdot \id
	\label{eq_959}
	\end{equation}
	in the sense of symmetric matrices.
	\item[(ii)] The function $\psi$, as well as each of its partial derivatives, grows at most polynomially at infinity.
\end{enumerate}

\begin{lemma} Let $\mu$ be an absolutely-continuous, log-concave probability measure in $\RR^n$ and let $0 < \eps < t$.
	Then there exists a regular, log-concave probability measure $\nu$ in $\RR^n$ such that
	$$ C_P(\nu) \geq C_P(\mu) - \eps \qquad \text{and} \qquad \| Cov(\mu) - Cov(\nu) \|_{op} < \eps. $$
	Moreover, if $\mu$ is $t$-uniformly log-concave, then $\nu$ is $(t - \eps)$-uniformly log-concave.
	\label{lem_approx}
\end{lemma}

The proof of this approximation lemma is deferred to the Appendix below.
Assume from now on that  $\mu$ is a probability measure on $\RR^n$ with a regular, log-concave probability density $\rho = e^{-\psi}$.
As in Klartag and Putterman \cite{KP}, we say that a function $f: \RR^n \rightarrow \RR$ has subexponential decay relative to $\rho$
if there exist $C, a > 0$ such that
\begin{equation} |f(x)| \leq \frac{C}{\sqrt{\rho(x)}} e^{-a|x|} \qquad \qquad \qquad (x \in \RR^n).
\label{eq_1030} \end{equation}
We say that a function $f: \RR^n \rightarrow \RR$ is $\mu$-{\it tempered} if it is smooth and if all of its partial derivatives
of all orders have subexponential decay relative to $\rho$.
A log-concave probability density decays exponentially at infinity. Therefore if a smooth function $f: \RR^n \rightarrow \RR$ grows at most polynomially at infinity,
as do all of its partial derivatives, then it is $\mu$-tempered. A $\mu$-tempered function is clearly in $L^2(\mu)$.
Write $\cF_{\mu}$ for the collection of all $\mu$-tempered functions on $\RR^n$.
The Laplace operator associated with $\mu$ is defined for $u \in \cF_{\mu}$ via
\begin{equation}  L u = L_{\mu} u = \Delta u - \nabla \psi \cdot \nabla u. \label{eq_1647} \end{equation}
For any $u,v \in \cF_{\mu}$ we have the integration by parts formula
\begin{equation} \int_{\RR^n} (L u) v d \mu = -\int_{\RR^n} \langle \nabla u, \nabla v \rangle d \mu
\label{eq_1047} \end{equation}
and the (integrated) Bochner formula
\begin{equation}
\int_{\RR^n} (L u)^2 d \mu = \int_{\RR^n} \| \nabla^2 u \|_{HS}^2 d \mu + \int_{\RR^n} \langle (\nabla^2 \psi) \nabla u, \nabla u \rangle d \mu,
\label{eq_1531}
\end{equation}
where $\| \nabla^2 u \|_{HS}$ is the Hilbert-Schmidt norm of the Hessian matrix $\nabla^2 u$.
Formulae (\ref{eq_1047}) and (\ref{eq_1531}) are proven by intergation by parts,
see e.g., Ledoux \cite[Section 2.3]{ledoux1}. The regularity of $\rho$ and the $\mu$-temperedness of $u,v$ are used in order to discard the boundary terms.
The integrated Bochner formula (\ref{eq_1531}) is related to the commutation relation $L(\nabla u) = \nabla L u + (\nabla^2 \psi) \nabla u$,
and is reminiscent of similar formulae in Riemannian geometry.

\vs
The operator $-L$ is essentially self-adjoint and positive semi-definite in $\cF_{\mu} \subseteq L^2(\mu)$ with a discrete spectrum (see e.g.,
\cite[Corollary 4.10.9]{BGL} and \cite[Proposition A.1]{KP}).

\begin{lemma} All eigenfunctions of $L_{\mu}$ are $\mu$-tempered.
	\label{lem_tempered}
\end{lemma}

The proof of Lemma \ref{lem_tempered} is discussed in the Appendix; as in \cite[Section 2]{KP} this proof is reduced to known results on exponential decay of eigenfunctions of Schr\"odinger
operators.
The minimal eigenvalue of $-L$ is $0$, which is a simple eigenvalue corresponding to a constant eigenfunction. We write
$$ \lambda = \lambda(\mu) > 0 $$
for the minimal non-zero eigenvalue of $-L$. The positive semi-definite operator $-L$ has no spectrum in the interval $(0, \lambda)$,
while $0$ is an eigenvalue of multiplicity one. It follows that for any $f \in \cF_\mu$ with $\int f d \mu = 0$,
\begin{equation}  \lambda \int_{\RR^n} f^2 d \mu \leq \int_{\RR^n} |\nabla f|^2 d \mu
\label{eq_1852} \end{equation}
with equality if and only if $f$ is a eigenfunction of $L$ corresponding to the first non-zero eigenvalue.
Inequality  (\ref{eq_1852}) is sometimes referred to as the Poincar\'e inequality for
the measure $\mu$. Indeed, the space of compactly-supported, smooth
functions is dense in the Sobolev space $H^1(\mu)$, see the Appendix of \cite{BK}, and hence
$$ \lambda(\mu) = 1 / C_P(\mu). $$
One way to  prove the log-concave Lichnerowicz inequality (\ref{eq_154_}) is to substitute the eigenfunction $f$
into the Bochner formula (\ref{eq_1531}) and obtain the inequality $\lambda^2 \geq t \lambda$
which is equivalent to (\ref{eq_154_}). The following is a quantitative, log-concave version of Corollary 1 from \cite{K_uncond}.

\begin{proposition} Let $f \in \cF_{\mu}$ be an eigenfunction of $-L$ corresponding to the eigenvalue $\lambda = \lambda(\mu)$
	and normalized so that $\| f \|_{L^2(\mu)} = 1$. Then,
	\begin{equation}  \left| \int_{\RR^n} \nabla f d \mu  \right|^2 \geq \frac{1}{\lambda} \int_{\RR^n} [(\nabla^2 \psi) \nabla f \cdot \nabla f] d \mu,
	\label{eq_1817}
	\end{equation}
	and
	\begin{equation}  \left| \int_{\RR^n} \nabla f  d \mu \right|^2 = \lambda^2 \left| \int_{\RR^n} f(x) x d \mu(x) \right|^2 \leq \lambda^2 \| Cov(\mu) \|_{op}.
	\label{eq_1818} \end{equation} \label{prop_1432}
\end{proposition}

\begin{proof} By the Bochner formula and the Poincar\'e inequality for $\partial^i f \ (i=1,\ldots,n)$,
	\begin{align} \nonumber \lambda^2 & = \int_{\RR^n} (L f)^2 d \mu = \int_{\RR^n} [(\nabla^2 \psi) \nabla f \cdot \nabla f] d \mu + \int_{\RR^n} \| \nabla^2 f \|_{HS}^2 d \mu
	\\ \nonumber  & \geq  \int_{\RR^n} [(\nabla^2 \psi) \nabla f \cdot \nabla f] d \mu + \lambda \left[ \int_{\RR^n} |\nabla f|^2 d \mu - \left| \int_{\RR^n} \nabla f d \mu \right|^2 \right]
	\\ & =  \int_{\RR^n} [(\nabla^2 \psi) \nabla f \cdot \nabla f] d \mu  + \lambda^2 - \lambda \left| \int_{\RR^n} \nabla f d \mu \right|^2.
	\label{eq_1731}
	\end{align}
	This implies (\ref{eq_1817}). As for (\ref{eq_1818}), note that for any $\theta \in S^{n-1}$, denoting $g(x) = \langle x, \theta \rangle$,
	\begin{equation}  \int_{\RR^n} \langle \nabla f, \theta \rangle  d \mu = \int_{\RR^n} \langle \nabla f(x), \nabla g(x) \rangle d \mu(x) = -\int_{\RR^n} (L f) g d \mu = \lambda \int_{\RR^n} f(x) \langle x, \theta \rangle d\mu(x).
	\label{eq_1458} \end{equation}
	Denote $E = \int_{\RR^n} \langle x, \theta \rangle d \mu(x)$. We use the facts that $\int f d \mu = 0, \int f^2 d \mu = 1$ and  the Cauchy-Schwartz inequality, and obtain
	$$ \left| \int_{\RR^n} \langle \nabla f, \theta \rangle  d \mu \right|^2 = \lambda^2 \left| \int_{\RR^n} f(x) (\langle x, \theta \rangle - E)  d\mu(x) \right|^2
	\leq \lambda^2 \int_{\RR^n} (\langle x, \theta \rangle - E)^2 d \mu(x). $$
	This implies (\ref{eq_1818}), since the last integral equals $\langle Cov(\mu) \theta, \theta \rangle$.
\end{proof}

We intuitively think of a function $f \in \cF_{\mu}$ as having a ``preferred direction'' when
$$ \left| \int_{\RR^n} \nabla f d \mu \right|^2 \gtrsim   \int_{\RR^n} \left| \nabla f \right|^2 d \mu. $$
The property that the eigenfunction $f$ from Proposition \ref{prop_1432} has a preferred direction
is  in the spirit of the ``hot spots conjecture'' \cite{JM}, as we learned
from David Jerison. The KLS conjecture suggests that
$$ \lambda(\mu) \cdot \| Cov(\mu) \|_{op} \geq c $$
for a universal constant $c > 0$. Since $\int |\nabla f|^2 d \mu = \lambda$ for the eigenfunction
$f$ from Proposition \ref{prop_1432}, we see from (\ref{eq_1818})
that if the eigenfunction has a preferred
direction, then the conclusion of the KLS conjecture holds.

\begin{proof}[Proof of Theorem \ref{prop_1944}]
	We may assume that $\mu$ is not supported in an affine  subspace $E \subsetneq \RR^n$
	(otherwise, we work in this subspace). Since $\mu$ is log-concave,
	this means that the measure $\mu$ is necessarily absolutely-continuous in $\RR^n$.
	An approximation argument based on Lemma \ref{lem_approx} shows that it suffices to prove the theorem for an absolutely-continuous,
	regular, log-concave probability measure $\mu$ on $\RR^n$.
	
	\vs
	Let $f$ be an eigenfunction of the operator $-L_{\mu}$ corresponding to the eigenvalue $\lambda = \lambda(\mu)$
	and normalized so that $\| f \|_{L^2(\mu)} = 1$. By combining (\ref{eq_1817}) and (\ref{eq_1818}) we obtain
	\begin{equation}  \int_{\RR^n} \left[ (\nabla^2 \psi) \nabla f \cdot \nabla f \right] d \mu \leq \lambda^3 \cdot \| Cov(\mu) \|_{op}.
	\label{eq_1935} \end{equation}
	Since $\mu$ is $t$-uniformly log-concave, we know that $\nabla^2 \psi(x) \geq t$
	for all $x \in \RR^n$. Hence  the left-hand side of (\ref{eq_1935}) is at least
	\begin{equation}  t \int_{\RR^n} |\nabla f|^2 d \mu = t \lambda. \label{eq_1502} \end{equation}
	Hence  (\ref{eq_1935}) implies that $\lambda^2 \| Cov(\mu) \|_{op} \geq t$. This proves the left-hand side inequality in (\ref{eq_1518}).
	For the right-hand side inequality, for any $\theta  \in S^{n-1}$ we set $g_{\theta}(x) = \langle x - b_{\mu}, \theta \rangle$
	where
	$$ b_{\mu} = \int_{\RR^n} x d \mu(x) \in \RR^n $$ is the barycenter of $\mu$.
	We now  use the Poincar\'e inequality
	(\ref{eq_1852}) as follows:
	$$ \| Cov(\mu) \|_{op} = \sup_{\theta \in S^{n-1}} Cov(\mu) \theta \cdot \theta = \sup_{\theta \in S^{n-1}} \int_{\RR^n} g_{\theta}^2 d \mu \leq \frac{1}{\lambda} \sup_{\theta \in S^{n-1}} \int_{\RR^n} |\nabla g_{\theta}|^2 d \mu = \frac{1}{\lambda} \leq t, $$
	where we used the log-concave Lichnerowicz inequality (\ref{eq_154_}) in the last passage.
\end{proof}

\section{Logarithmic bound for the KLS constant}
\label{sec3}

We use the notation from \cite[Section 4]{KP}.
Let $\mu$ be a probability measure on $\RR^n$ with a regular, log-concave probability density $\rho = e^{-\psi}$.
For $t \geq 0$ and $\theta \in \RR^n$ we denote
\begin{equation}  p_{t, \theta}(x) = \frac{1}{Z(t,\theta)} e^{\langle \theta, x \rangle - t |x|^2/2} \rho(x) \qquad \qquad \qquad (x \in \RR^n)
\label{eq_1015} \end{equation}
where $Z(t, \theta) = \int_{\RR^n} e^{\langle \theta, x \rangle - t |x|^2/2} \rho(x) dx$. The barycenter and
covariance matrix of the probability density $p_{t, \theta}$ are denoted by
$$ a(t,\theta) = \int_{\RR^n} x p_{t, \theta}(x) \, dx \in \RR^n$$
and
$$ A(t,\theta) = \int_{\RR^n} (x \otimes x) p_{t, \theta}(x) \, dx \, - \, a(t, \theta) \otimes a(t, \theta) \in \RR^{n \times n}, $$
where $x \otimes x = (x_i x_j)_{i,j=1,\ldots,n} \in \RR^{n \times n}$.
Write $\mu_{t, \theta}$ for the probability measure whose density is the regular, log-concave probability density $p_{t, \theta}$.
We abbreviate $$ \lambda(t, \theta) = \lambda(\mu_{t, \theta}), $$
the first non-zero eigenvalue of the operator $-L_{\mu_{t, \theta}}$. Since $\rho$ is log-concave, the probability measure $\mu_{t, \theta}$ is $t$-uniformly log-concave,
as we see from formula (\ref{eq_1015}). We may therefore apply the improved log-concave Lichnerowicz inequality, which is
Theorem \ref{prop_1944} above, and obtain the bound
\begin{equation}
\lambda(t, \theta)  \geq \sqrt{ \frac{t}{\| A(t, \theta) \|_{op}} } \geq t.
\label{eq_1017}
\end{equation}
Given $f \in L^1(\mu)$ we write
$$ M_f(t, \theta) = \int_{\RR^n} f(x) p_{t, \theta}(x) dx.
$$
This is a smooth function of $\theta \in \RR^n$, and by differentiating under the integral sign
we obtain
\begin{align}  \nonumber \nabla_{\theta} M_f(t, \theta) & = \int_{\RR^n} (x - a(t, \theta)) f(x) p_{t, \theta}(x) dx
\\ & = \int_{\RR^n} (x - a(t, \theta)) (f(x) - M_f(t, \theta)) p_{t, \theta}(x) dx.
\label{eq_1028}
\end{align}
The ``tilt process'' is the stochastic process $(\theta_t)_{t \geq 0}$ attaining values in $\RR^n$ and defined via the stochastic differential equation
\begin{equation}  \theta_0 = 0 , \quad d \theta_t = d W_t + a(t, \theta_t) dt, \label{eq_910} \end{equation}
where $(W_t)_{t \geq 0}$ is a standard Brownian motion in $\RR^n$ with $W_0 = 0$. The existence and uniqueness of a strong solution
to (\ref{eq_910}) are proven via a standard argument (see, e.g. Chen \cite{chen}).
As explained in \cite[Section 4]{KP}, the process $(\theta_t)_{t \geq 0}$ coincides
in law with the process $(t X + W_t)_{t \geq 0}$
where $X$ is a random vector with law $\mu$, independent of the Brownian motion $(W_t)_{t \geq 0}$.
Setting $p_t(x) = p_{t, \theta_t}(x)$ and $a_t = a(t, \theta_t)$ we obtain from the It\^o formula that
$$ d p_t(x) = p_t(x) \langle x - a_t, d W_t \rangle \qquad \qquad \qquad (x \in \RR^n)
$$
with $p_0(x) = \rho(x)$. This stochastic process is central  in the theory of non-linear filtering,
see e.g. Chiganski \cite[Chapter 6]{chiganski} and references therein for its analysis and history.
Its usefulness for proving isoperimetric inequalities and functional inequalities,  for bounding mixing times, and for analyzing probability measures with convexity properties
was realized by Eldan \cite{Eldan1},
Lee and Vempala \cite{LV}, Chen \cite{chen} and others.
The process $(p_t(x))_{t \geq 0}$ is a martingale with respect to the filtration induced by the Brownian motion,
and in particular
\begin{equation} \EE p_t(x) = p_0(x) = \rho(x) \qquad \qquad (t \geq 0, x \in \RR^n). \label{eq_1048}
\end{equation}
Abbreviate $\lambda_t = \lambda(t, \theta_t), A_t = A(t, \theta_t)$ and note that $\lambda_0 = \lambda(\mu)$.

\begin{lemma} For any $t > 0$ and $f \in L^2(\mu)$,
	$$ \EE Var_{p_t}(f) \leq  Var_{p_0}(f) \leq \left( 2 + \frac{t}{\lambda_0} \right)  \EE Var_{p_t}(f) $$
	where we write $Var_{p_t}(f) = \int_{\RR^n} f^2 p_t - (\int_{\RR^n} f p_t)^2$.
	\label{lem_1131}
\end{lemma}

\begin{proof} Set $M_t = M_f(t, \theta_t) = \int f p_t$.
	It follows from (\ref{eq_1048}) and Fubini's theorem that
	\begin{equation}  Var_{p_0}(f) = \int_{\RR^n} (f - M_0)^2 p_0 = \EE \int_{\RR^n} (f - M_0)^2 p_t = \EE \int_{\RR^n} (f - M_t)^2 p_t + \EE (M_t-M_0)^2.
	\label{eq_1055} \end{equation}
	The first summand in (\ref{eq_1055}) equals $\EE Var_{p_t}(f)$, which is evidently at most $Var_{p_0}(f)$.
	As for the second summand in (\ref{eq_1055}), it equals
	$$ Var(M_t) = Var( M_f(t, \theta_t) ), $$
	where the random vector $\theta_t$ has the law of $t X + W_t$,
	where $X$ is distributed according to $\mu$, and where $W_t$ is a Gaussian random vector in $\RR^n$ of mean
	zero and covariance matrix $t \cdot \id$ that is independent of $X$. By the subadditivity property of the Poincar\'e constant (see e.g. Courtade \cite{cour}),
	\begin{equation}  C_P(\theta_t) = C_p(t X + W_t) \leq C_p(t X) + C_p(W_t) = t^2/ \lambda_0 + t = \alpha t
	\label{eq_1754} \end{equation}
	for $\alpha = 1 + t / \lambda_0$. According to (\ref{eq_1028}) and to the Cauchy-Schwartz inequality, for any $\theta \in \RR^n$,
	\begin{align*}
	|\nabla_{\theta} M_f(t, \theta)| & = \sup_{\eta \in S^{n-1}} \int_{\RR^n} \langle x - a(t, \theta), \eta \rangle (f(x) - M_f(t, \theta)) p_{t, \theta}(x) dx
	\\ & \leq \sqrt{ \| A(t, \theta) \|_{op} \cdot Var_{p_{t, \theta}}(f)  }.
	\end{align*}
	Hence, by (\ref{eq_1754}) and the Poincar\'e inequality for the random vector $\theta_t$,
	\begin{equation}  Var( M_f(t, \theta_t) ) \leq \alpha t \cdot \EE |\nabla_{\theta} M_f(t, \theta_t)|^2
	\leq \alpha t \cdot \EE \left[ \| A_t \|_{op}  Var_{p_t}(f) \right] \leq \alpha  \cdot \EE Var_{p_t}(f), \label{eq_1110} \end{equation}
	where we used the second inequality in (\ref{eq_1017}) in the last passage, which is equivalent to the bound $\| A_t \|_{op} \leq 1/t$. Thus, by (\ref{eq_1055}) and (\ref{eq_1110}),
	$$    Var_{p_0}(f) \leq \EE Var_{p_t}(f) + \alpha \cdot \EE Var_{p_t}(f) = (1 + \alpha) \EE Var_{p_t}(f).
	$$
\end{proof}

\begin{corollary} For any $A > 0$ and  $0 < t < A \lambda_0$,
	\begin{equation}  \lambda_0^{-1} \leq C (A+1) \EE \lambda_t^{-1} \leq C (A+1)  \frac{\EE \sqrt{\| A_t \|_{op}}}{\sqrt t},
	\label{eq_1828} \end{equation}
	where $C > 0$ is a universal constant. \label{cor_1804}
\end{corollary}

\begin{proof} E. Milman's theorem \cite{mil} states that
	\begin{equation} c \lambda_0^{-1} \leq \sup_{\vphi \textrm{ is 1-Lipschitz}} Var_{\mu}(\vphi)
	\label{eq_1045} \end{equation}
	where the supremum runs over all $1$-Lipschitz functions $\vphi: \RR^n \rightarrow \RR$,
	and where $c > 0$ is a universal constant. See \cite{K_needle} for a proof of this theorem that uses {\it needle decompositions} rather
	than the regularity theory of the isoperimetric problem.
	Let $f: \RR^n \rightarrow \RR$ be a $1$-Lipschitz function with
	\begin{equation} Var_{\mu}(f)  \geq c \lambda_0^{-1} / 2.
	\label{eq_1047_} \end{equation}
	By Lemma \ref{lem_1131} and the Poincar\'e inequality,
	\begin{equation}
	Var_{\mu}(f) = Var_{p_0}(f) \leq (2 + A) \EE Var_{p_t}(f) \leq (2 + A)  \EE \lambda_t^{-1} \int_{\RR^n} |\nabla f|^2 p_t \leq (2 + A) \EE \lambda_t^{-1}. \label{eq_1056}
	\end{equation}
	The left-hand side inequality in  (\ref{eq_1828}) follows from (\ref{eq_1047_}) and (\ref{eq_1056}).
	The right-hand side inequality in  (\ref{eq_1828})
	follows from (\ref{eq_1017}).
\end{proof}

Consider the covariance process $A_t = A(t, \theta_t)$ defined for $t \geq 0$. Since $\mu$ is isotopic we have $A_t = \id$ for $t = 0$.
One of the features of the covariance process is the fact that for $0 < t \leq c \psi_n^{-2} / \log n$,
\begin{equation}
\EE \| A_t \|_{op} \leq C.
\label{eq_1732}
\end{equation}
A proof of inequality (\ref{eq_1732}) appears in \cite{k_chen}. We remark that Corollary 5.4 in Klartag and Lehec \cite{KL} states that (\ref{eq_1732}) holds true
whenever $t \leq c \kappa_n^{-2} / \log n$, where
$$ \kappa_n = \sup_X \EE \| X_1 (X \otimes X) \|_{HS}^2 $$
and the supremum runs over all isotropic, log-concave random vectors $X = (X_1,\ldots,X_n) \in \RR^n$.
The quantity $\kappa_n$ was used by Eldan \cite{Eldan1}, and it satisfies
\begin{equation}
\kappa_n^2  \leq 4 \sup_{X} C_P(X) \leq C \psi_n^2
\label{eq_1737}
\end{equation}
for a universal constant $C >0$, where the supremum runs over all isotropic, log-concave random vectors $X \in \RR^n$.
Indeed, the right-hand side inequality in (\ref{eq_1737}) follows from (\ref{eq_1103}) and (\ref{eq_1104}) above.
In order to prove the left-hand side inequality in (\ref{eq_1737}) we argue as follows:
Let $X$ be an isotropic, log-concave
random vector, and denote $B = \EE X_1 X \otimes X$ and $f(x) = \langle B x, x \rangle$. Then,
\begin{align*} \| B \|_{HS}^2 & = \EE X_1 \langle BX, X \rangle \leq \sqrt{ \EE X_1^2} \cdot \sqrt{ Var( f(X) ) }
\leq \sqrt{C_P(X) \cdot \EE |\nabla f(X)|^2 } \\ & = 2  \sqrt{ C_P(X) \cdot \EE |BX|^2} = 2 \sqrt{ C_P(X) }  \| B \|_{HS}.
\end{align*}
This implies the left-hand side inequality in (\ref{eq_1737}).

\begin{proof}[Proof of Theorem \ref{thm2}]
	Let $\mu$ be an isotropic, log-concave probability measure in $\RR^n$ with
	\begin{equation}  \psi_{\mu} \geq \frac{ \psi_n}{2}. \label{eq_1814} \end{equation}
	As we discussed above, for $ t = c \psi_n^{-2} / \log n$,
	\begin{equation}  \EE \sqrt{ \| A_t \|_{op} } \leq \sqrt{ \EE \| A_t \|_{op} } \leq C.
	\label{eq_1831} \end{equation}
	Since $\lambda_0 = 1 / C_P(\mu)$, from (\ref{eq_1103}) and (\ref{eq_1104}) we know that
	\begin{equation}  t \leq c' \psi_n^{-2} \leq \tilde{c} \psi_{\mu}^{-2} \leq \bar{c} \lambda_{0}. \label{eq_1859} \end{equation}
	We now apply Corollary \ref{cor_1804}, together with (\ref{eq_1831}) and (\ref{eq_1859}), and obtain
	\begin{equation}  \lambda_0^{-1} \leq \frac{C}{\sqrt{t}} \leq \tilde{C} \psi_n \sqrt{\log n}.  \label{eq_1122} \end{equation}
	Consequently, by using (\ref{eq_1103}), (\ref{eq_1814}) and (\ref{eq_1122}),
	$$  \psi_n^2 \leq 4 \psi_{\mu}^2 \leq \tilde{C} \lambda_0^{-1} \leq \bar{C} \psi_n \sqrt{\log n}. $$
	This implies that $\psi_n \leq C \sqrt{\log n}$.
\end{proof}

\section{Remarks}

One point in the proof above which seems counter-intuitive
is the fact that if the eigenfunction $f$ satisfies
\begin{equation}  \left| \int_{\RR^n} f(x) x d \mu(x) \right|^2 = o \left( \| Cov(\mu) \|_{op} \right),  \label{eq_1432} \end{equation}
then our estimates become {\it better}. Here $\mu$ is an isotropic, regular, log-concave
probability measure in $\RR^n$, and $f \in L^2(\mu)$ is an eigenfunction satisfying $Lf = -\lambda(\mu) f$ and $\| f \|_{L^2(\mu)} = 1$.
Intuitively, strong correlation between the eigenfunction and a linear function should
improve estimates towards the KLS conjecture, rather than the other way around. One could speculate that analysis of
the $H^{-1}$-norm of linear functionals, which played a crucial role in Klartag and Lehec \cite{KL}, could help us understand
this point. The $H^{-1}(\mu)$-norm
may be defined, for  $f \in L^2(\mu)$ with $\int_{\RR^n} f d \mu = 0$, via
$$ \| f \|_{H^{-1}(\mu)} := \sup \left \{ \int_{\RR^n} f g d \mu \, ; \, g \in L^2(\mu) \textrm{ is locally Lipschitz with } \int_{\RR^n} |\nabla g|^2 d \mu \leq 1 \right \}.  $$
See \cite{BK, K_uncond} for basic properties of the $H^{-1}(\mu)$-norm and for the inequality $$ \lambda(\mu) \cdot \| f \|_{H^{-1}(\mu)}^2 \leq \| f \|_{L^2(\mu)}^2. $$
Here is a certain version of the improved log-concave Lichnerwicz inequality:

\begin{proposition} Let $t > 0$ and assume that $\mu$ is a probability measure in $\RR^n$ that is regular and $t$-uniformly log-concave.
	Assume furthermore that $\mu$ is centered, i.e., $b_{\mu} = \int_{\RR^n} x d \mu(x) = 0$. Denote
	$$ R = \sup_{\theta \in S^{n-1}} \left \| \langle \cdot, \theta \rangle \right \|^2_{H^{-1}(\mu)}. $$
	Then $ \lambda(\mu) \geq (t / R)^{1/3}$.
\end{proposition}

\begin{proof} We use the notation from the proof of Theorem  \ref{prop_1944}. From (\ref{eq_1817}) and (\ref{eq_1818}) we obtain
	\begin{align*}
	\int_{\RR^n} \left[ (\nabla^2 \psi) \nabla f \cdot \nabla f \right] d \mu & \leq \lambda^3 \left| \int_{\RR^n} f(x) x d \mu(x) \right|^2 = \lambda^4 \sup_{\theta \in S^{n-1}} \left| \int_{\RR^n} \frac{f(x)}{\sqrt{\lambda}} \langle x, \theta \rangle d \mu(x) \right|^2 \\ & \leq \lambda^4 \sup_{\theta \in S^{n-1}} \| \langle \cdot, \theta \rangle \|_{H^{-1}(\mu)}^2 = \lambda^4 R.
	\end{align*}
	Hence, from (\ref{eq_1502}) we obtain $t \lambda \leq \lambda^4 R$, which yields $\lambda \geq (t / R)^{1/3}$.
\end{proof}

\begin{remark}
	As explained in \cite{BK, KL}, for any $f \in L^2(\mu)$ with $\int_{\RR^n} f d \mu = 0$ we have
	$$ \| f \|^2_{H^{-1}(\mu)} = -\int_{\RR^n} (L^{-1} f) f d \mu, $$
	where here $L$ is the unique self-adjoint extension of the operator defined on $\cF_{\mu}$ via (\ref{eq_1647}).
	Observe that $L x = -\nabla \psi$, and that the vector-valued functions  $\nabla \psi$ and $x$ are $\mu$-tempered. Therefore,
	\begin{equation}
	- \int_{\RR^n} \langle L^{-1} (\nabla \psi), \nabla \psi \rangle d \mu = \int_{\RR^n} \langle x, \nabla \psi \rangle e^{-\psi} = -\sum_{i=1}^n \int_{\RR^n}  x_i \partial_i (e^{-\psi}) = n \cdot \int_{\RR^n} e^{-\psi} = n.
	\label{eq_1636} \end{equation}
	The $H^{-1}$-inequality \cite{BK, K_uncond} states that for any Lipschitz function $f: \RR^n \rightarrow \RR$ with $\int \nabla f d \mu =0$,
\begin{equation}  Var_{\mu}(f) \leq \sum_{i=1}^n \left \| \partial^i f \right \|_{H^{-1}(\mu)}^2 =: \left \| \nabla f \right \|_{H^{-1}(\mu)}^2, \label{eq_2311} \end{equation}
	where $Var_{\mu}(f) = \int f^2 d \mu - (\int f d \mu)^2$. Hence, from (\ref{eq_1636}) and (\ref{eq_2311})
	we obtain an alternative proof of the varentropy inequality of Bobkov and Madiman \cite{BM} and Nguyen \cite{Nguyen}, which is the inequality:
	$$ Var_{\mu}(\psi) \leq \| \nabla \psi \|_{H^{-1}(\mu)}^2 = n. $$
\end{remark}

\bigskip There is more than one way to deduce Theorem \ref{thm2} from the improved log-concave Lichnerowicz inequality. The following variant
of the integrated Bochner formula shows  another connection between  the Laplace operator
associated with $p_0$ and the one associated with $p_t$, in addition
to the connection established in Corollary \ref{cor_1804} above by using $1$-Lipschitz functions.

\begin{proposition}[``Localized Bochner formula''] Let $\mu$ be a probability measure on $\RR^n$ with a regular, log-concave probability density $\rho = e^{-\psi}$.
	Let $u \in \cF_{\mu}, t > 0$, and consider the random probability density $p_t$
	defined in Section \ref{sec3}. Write $\mu_t$ for the regular, log-concave probability measure on $\RR^n$
	whose density is $p_t$ and abbreviate $L_t = L_{\mu_t}$ and $L = L_{\mu}$. Then with probability one, $u \in \cF_{\mu_t}$ and
	$$ \int_{\RR^n} (L u)^2 d \mu + t \int_{\RR^n} |\nabla u|^2 d \mu = \EE \int_{\RR^n} (L_t u)^2 d \mu_t. $$
	\label{prop_1034}
\end{proposition}

\begin{proof} We know that $d \mu_t / d \mu$ is a Gaussian density, hence the fact that $u \in \cF_{\mu}$ implies
	that $u \in \cF_{\mu_t}$. Write $p_t = e^{-\psi_t}$ and apply the Bochner formula for the measure $\mu_t$ to obtain
	\begin{equation} \int_{\RR^n} (L_t u)^2 d \mu_t = \int_{\RR^n} \| \nabla^2 u \|_{HS}^2 d \mu_t + \int_{\RR^n} \langle (\nabla^2 \psi_t) \nabla u, \nabla u \rangle d \mu_t. \label{eq_1200} \end{equation}
	From formula (\ref{eq_1015}) we know that
	$$ \nabla^2 \psi_t(x) = \nabla^2 \psi(x) + t \cdot \id \qquad \qquad \qquad (x \in \RR^n). $$
	Therefore, by combining (\ref{eq_1200}) with (\ref{eq_1048}) and Fubini's theorem,
	\begin{align*} \EE \int_{\RR^n} (L_t u)^2 d \mu_t & = \EE \int_{\RR^n} \| \nabla^2 u \|_{HS}^2 d \mu_t + \EE \int_{\RR^n} \langle (\nabla^2 \psi) \nabla u, \nabla u \rangle d \mu_t
	+ t \cdot \EE \int_{\RR^n} |\nabla u|^2 d \mu_t
	\\ & = \int_{\RR^n} \| \nabla^2 u \|_{HS}^2 d \mu + \int_{\RR^n} \langle (\nabla^2 \psi) \nabla u, \nabla u \rangle d \mu + t \cdot \int_{\RR^n} |\nabla u|^2 d \mu.
	\end{align*}
	The proposition now follows from the Bochner formula (\ref{eq_1531}) applied for the measure $\mu$.
\end{proof}

\section{Appendix}

Write $\gamma_s$ for the density of a Gaussian random vector of mean zero and covariance $s \cdot \id$ in $\RR^n$,
i.e., $\gamma_s(x) = (2 \pi s)^{-n/2} \exp(-|x|^2 / (2s))$ for $x \in \RR^n$. The following lemma allows
us to interchange Gaussian convolution and multiplication by a Gaussian density.

\begin{lemma}  For $f: \RR^n \rightarrow \RR$ and $r > 0$ denote $S_r f(x) = r^n f(r x)$, the scaling of $f$ when viewed as a density. Let $s, t > 0$
	and set
	$$ p = \frac{st}{s+t}, \qquad q = \frac{t^2}{s+t} \qquad \text{and} \qquad r = \frac{t}{s+t}. $$
	Then for any measurable function $f: \RR^n \rightarrow \RR$ such that $f \gamma_t$ is integrable,
	$$
	(f \gamma_t) * \gamma_s = S_r \left[ (f * \gamma_p)  \gamma_q  \right].
	$$
	\label{lem_1721}
\end{lemma}

\begin{proof} Let $x \in \RR^n$ and set $X = px / s$. Then,
	\begin{align*}
	[(f \gamma_t) * \gamma_s](x) & = (2 \pi)^{-n} (st)^{-n/2} \int_{\RR^n} f(y) e^{-\frac{|y|^2}{2t} - \frac{|x-y|^2}{2s}} dy
	\\ & = (2 \pi)^{-n} (st)^{-n/2} e^{\frac{|x|^2 (-1/s + p/s^2)}{2}} \int_{\RR^n} f(y) e^{-\frac{|y - px/s|^2}{2p} } dy
	\\ & = (2 \pi)^{-n} (st)^{-n/2} e^{-\frac{|X|^2}{2q}} \int_{\RR^n} f(y) e^{-\frac{|y - X|^2}{2p} } dy
	= r^n \left[ (f * \gamma_p) \gamma_q \right](X).
	\end{align*}
\end{proof}

\begin{corollary} Let $s, t > 0$, and let $\mu$ be an absolutely-continuous
	probability measure on $\RR^n$ which is $1/t$-uniformly log-concave.
	Then the convolution $ \mu * \gamma_s$ is $1/(t+s)$-uniformly log-concave.
	\label{cor_1008}
\end{corollary}

\begin{proof} We say that $f$ is more log-concave than $\gamma_t$
	if $f / \gamma_t$ is log-concave, i.e., if $f$ is $1/t$-uniformly log-concave.
	Write $\rho \gamma_t$ for the density of $\mu$, thus $\rho$ is log-concave.
	Define $p,q,r$ as in Lemma \ref{lem_1721}. By the Pr\'ekopa-Leindler inequality,
	the function $\rho * \gamma_p$ is log-concave, where the convolution is well-defined by Lemma \ref{lem_1721}.  The density of $\mu * \gamma_s$ is  $(\rho \gamma_t) * \gamma_s$, which by Lemma \ref{lem_1721} is more log-concave than
	$$ S_r \gamma_q = \gamma_{s+t}. $$
	Indeed, we compute that
	$$ S_r \gamma_q(x) = (2 \pi q/r^2)^{-n/2} \exp(-|r x|^2/(2 q)) = \gamma_{q / r^2} = \gamma_{s+t}. $$
\end{proof}

\begin{proof}[Proof of Lemma \ref{lem_approx}]
	Let $0 < \delta < 1$ be a sufficiently small number, and write $\rho$ for the log-concave density of $\mu$.
	The  measure $\nu = \nu_{\delta}$ is obtained by convolving $\rho$
	with $\gamma_{\delta}$, then multiplying the resulting  density by $\gamma_{1/\delta}$, and then normalizing to obtain a probability density. The Prekop\'a-Leindler inequality shows that $\nu$ is a log-concave measure, and
	$\nu$ clearly has a smooth, positive
	density $e^{-\psi_{\delta}}$. We claim that
	\begin{equation} \delta \leq \nabla^2 \psi_{\delta}(x) \leq \delta + 1/\delta  \qquad \qquad \text{for all} \ x \in \RR^n.
	\label{eq_947} \end{equation}
	The left-hand side inequality in (\ref{eq_947}) follows from the fact that
	$\nu$ is obtained by multiplying a log-concave density by $\gamma_{1/\delta}$ and then normalizing.
	For the right-hand side inequality, we write
	\begin{equation}  e^{-\psi_{\delta}(x)} = e^{-(\delta + 1 / \delta) |x|^2 / 2 } \int_{\RR^n} e^{\frac{\langle x,y \rangle}{\delta} -\frac{|y|^2}{2 \delta}} \rho(y) d y / Z_\delta \label{eq_952} \end{equation}
	where $Z_{\delta} > 0$ is a normalizing constant. The logarithm of the integral in (\ref{eq_952}) is a convex function of $x$, leading
	to the right-hand side inequality in (\ref{eq_947}). The fact that the function $\psi_{\delta}$,
	as well as each of its partial derivatives, grows at most polynomially at infinity
	is proven in \cite[Lemma 2.2]{KP}. Thus $\nu_{\delta}$ is a regular, log-concave probability measure for any $0 < \delta < 1$.
	Next, consider a smooth function $f$ with
	$$ \int_{\RR^n} f^2 d \mu - \left(\int_{\RR^n} f d \mu \right)^2 \geq (C_P(\mu) - \eps/2) \int_{\RR^n} |\nabla f|^2 d \mu. $$
	By the dominated convergence theorem we may replace the integrals with respect to $\mu$ by integrals with respect to $\nu_{\delta}$
	while accumulating only a tiny error. This shows that for a sufficiently small $\delta > 0$,
	$$ C_P(\nu_{\delta}) \geq C_P(\mu) - \eps. $$  The dominated
	convergence theorem also shows that the covariance matrix of $\nu = \nu_{\delta}$ can be made arbitrarily close to that of $\nu$, for a sufficiently small
	$\delta > 0$. Finally, if $\mu$ is $t$-uniformly log-concave, then Corollary \ref{cor_1008} implies
	that $\nu_{\delta}$ is $[\delta + t / (1 +t \delta)]$-uniformly log-concave.
\end{proof}

\begin{proof}[Proof of Lemma \ref{lem_tempered}]
	Let $\vphi_k \in L^2(\mu)$ be an eigenfunction of $L_{\mu}$ corresponding to some eigenvalue $-\lambda_k$. Since $L$ is an elliptic operator, the function $\vphi_k$ is smooth.
	As in Remark 2.7 in Klartag and Putterman \cite{KP}, for any $k \geq 0$ the smooth function $\tilde{\vphi}_k = \vphi_k \cdot \sqrt{\rho} \in L^2(\RR^n)$ is an eigenfunction
	of the Schr\"odinger operator $-\Delta + V$ for $$ V = |\nabla \psi|^2 / 4 - \Delta \psi / 2, $$
	where we recall that $\rho = e^{-\psi}$.  The potential $V(x)$ is a smooth
	function that tends to infinity as $x \rightarrow \infty$. In fact, since $\mu$ is regular then
	$$ V(x) \geq \eps |x| - 1/\eps $$ for some $\eps > 0$ depending on $\mu$.
	It is known  that any eigenfunction $\tilde{\vphi}_k$ of the Schr\"odinger operator $-\Delta + V$  decays exponentially at infinity,
	see Agmon \cite[Chapter 5]{agmon}, Carmona \cite{carmona}, Combes and Thomas \cite{CT}, O'Connor \cite{O} or Reed and Simon \cite[Theorem XIII.70]{RS}.
	Moreover,
	since $\mu$ is regular, the potential $V$ and each of its partial derivatives
	grows at most polynomially at infinity. Hence both  functions $\tilde{\vphi}_k$ and $\Delta \tilde{\vphi}_k = (V - \lambda_k) \tilde{\vphi}_k$ decay exponentially at infinity.
	This implies that each partial derivative
	of $\tilde{\vphi}_k$ decays exponentially at infinity;
	to see this, consider a unit ball centered at a faraway point $x \in \RR^n$, and express $\tilde{\vphi}_k$
	via Green's representation for solutions to the Poisson equation in this unit ball (e.g. \cite[Section 2.2.4]{evans}).  For any multi-index $\alpha = (\alpha_1,\ldots,\alpha_n)$
	of non-negative integers, in the notation of \cite[Section 2]{KP},
	\begin{equation} \partial^{\alpha} \vphi_k =  \partial^{\alpha} (\tilde{\vphi}_k e^{\psi/2})
	= \sum_{0 \leq \beta \leq \alpha} {\alpha \choose \beta} \partial^{\beta} \tilde{\vphi}_k \cdot \partial^{\alpha - \beta} e^{\psi/2},
	\label{eq_1037} \end{equation}
	where we used the Leibnitz rule. Each of the summands in (\ref{eq_1037}) is some coefficient multiplied by the product of three factors: the exponentially decaying function
	$\partial^{\beta} \tilde{\vphi}_k$, a certain expression involving derivatives of $\psi$ that grows at most polynomially, and
	the function $e^{\psi / 2}$. This implies that $\vphi_k \in \cF_\mu$.
\end{proof}


\end{document}